%% file: Polygonal_Perspective.tex
\documentclass[12pt]{amsart}
\usepackage{amscd,amssymb,epsfig}
\usepackage{graphics,color}

\input epsf
\newtheorem{theorem}{Theorem}[section]
\newtheorem{lemma}[theorem]{Lemma}

\theoremstyle{definition}

\theoremstyle{remark}

\numberwithin{figure}{section}
\numberwithin{table}{section}

\newcommand{\lra}{\longrightarrow}
\newcommand{\ra}{\rightarrow}

\providecommand{\abs}[1]{\lvert#1\rvert}

\providecommand{\norm}[1]{\lVert#1\rVert}

\newcommand{\Teich}{Teichm\"uller }
\newcommand{\Poin}{Poincar\'e }

\def\ra{{\rightarrow}}
\def\lra{{\longrightarrow}}

\def\S1g{{\Sigma_{g,1}}}
\def\T1g{{\mathcal T}_{g,1}}
\def\M1g{{MC}_{g,1}}

\def\I1g{{\mathcal I}_{g,1}}

\newcommand\Aut{\operatorname{Aut}}

\addtocounter{MaxMatrixCols}{5}

\newcommand{\mb}{\mathbf}
\newcommand{\mc}{\mathcal}
\newcommand{\mf}{\mathfrak}

\newcommand{\fttw}[2]{\begin{array}{c}\vspace{-.05in}\includegraphics[#1]{#2}\end{array}}

\newcommand{\fttexw}[2]{\begin{array}{c}\resizebox{#1}{!}{\input{#2.pstex_t}}\end{array}}

\begin{document}

\title[Fatgraph Nielsen reduction and the chord slide groupoid]
{A Polygonal Perspective of Nielsen Reduction and the Chord Slide Groupoid$^\dagger$}

\thanks{$\dagger$ This short survey (minus minor corrections) was written for and published in the abstracts for the 57th Topology Symposium which took place in Okayama, JP \emph{2010}.  It summarizes the contents of a talk of the same title which was given by the author at that conference on \emph{8/12/2010}.  }

\author{Alex James Bene}
\address{IPMU, Tokyo University\\
5-1-5 Kashiwanoha\\
  Kashiwa 277-8568  \\
  Japan \\}
\email{alex.bene{\char'100}ipmu.jp}

%\subjclass{Primary 20F38, 05C25; Secondary   20F34, 57M99,  32G15, 20F05, 20F06,  20F99}

%\keywords{mapping class groups, Ptolemy groupoid,   fatgraphs, ribbon graphs, chord diagrams, Nielsen reduction }

%\subjclass{Primary 20F38, 05C25; Secondary   20F34, 57M99,  32G15, 14H10,  20F99}

%*%*%*%*%*%*%*
%%*%*%*%*%*%*%*
%%%*%*%*%*%*%*%*
\begin{abstract}
Nielsen reduction is an algorithm which decomposes any automorphism of a free group into a product of elementary Nielsen transformations.   While this may be applied to a mapping class of a surface $S_{g,1}$ with one boundary component, the resulting decomposition in general will not have a topological interpretation.  In this survey, we discuss a variation called fatgraph Nielsen reduction which decomposes such a mapping class into elementary Nielsen transformations interpreted as rearrangements of  polygon domains for $S_{g,1}$ described by  systems of arcs in $S_{g,1}$.  These elementary moves generate the chord slide groupoid of $S_{g,1}$,  which we survey and  describe in terms of generators and relations.

\end{abstract}

\maketitle

%\tableofcontents

%%%%%%%%%%%%%%%%%%%%%%%
\section{Introduction}

Let us begin by briefly recalling Nielsen reduction for an automorphism of a free group  $F_n$ on $n$ free generators $\{x_i\}_{i=1}^n$.  An \emph{elementary Nielsen transformation} with respect to these generators is an automorphism of $F_n$ given by either permuting two generators: $x_i\mapsto x_j$, inverting some generator: $x_i\mapsto \bar x_i$, or multiplying  some generator by another: $x_i\mapsto x_ix_j$ for $j\neq i$.  \emph{Nielsen reduction} is an algorithm which  applies basic cancellation theory to decomposes every $f\in \Aut(F_n)$ into a product of these elementary transformations \cite{MKS}.  Roughly, the algorithm  proceeds by continuously applying elementary transformations which  reduce the total word length of the generating set $\{f(x_i)\}_{i=1}^n$ with respect to the $x_i$'s %original generators  $\{x_i\}_{i=1}^n$, 
whenever possible,  with the final goal of obtaining $\{x_i\}_{i=1}^n$.   In the event that no length-reducing transformation is available, 
a lexicographical ordering of $F_n$ is used  to ensure progress is made  towards this final goal.  % when no word length-reducing elementary transformation is available.

Let $S_{g,1}$ be a surface of genus $g$ with one boundary component, and let $\pi=\pi_1(S_{g,1},p)$ be its fundamental group with respect to a basepoint $p\in \partial S_{g,1}$ on the boundary.  
The mapping class group $\mc{M}_{g,1}$ of $S_{g,1}$  acts on $\pi$ and can be identified with  the subgroup of $\Aut(\pi)$ which preserves the element  $\partial S_{g,1}\in \pi$ representing the boundary \cite{Zieschang}.
  Since $\pi\cong F_{2g}$ is a free group, every mapping class $\varphi\in \mc{M}_{g,1}$ can be decomposed via Nielsen reduction; however, this decomposition has no obvious topological interpretation.
 
In this short survey, we will discuss a variation of the above algorithm for mapping classes called fatgraph Nielsen reduction which was introduced in \cite{bene2}.  This algorithm  has a simple topological interpretation in terms of polygon domains of $S_{g,1}$ (see the next section) and in fact produces a sequence  of elementary moves, called CS moves, relating any two polygon domains of $S_{g,1}$, not only those differing by the action of a mapping class.  

Motivated by this application to Nielsen reduction, we introduce the \emph{chord slide groupoid} $\mf{CS}_{g,1}$,   the groupoid naturally generated by CS moves, which  can be thought of as a groupoid  laying somewhere ``between'' the mapping class group $\mc{M}_{g,1}$ and the full automorphism group $\Aut(\pi)$ (as discussed at the end of Section \ref{sect:CS}).  

 %sort of middle  ground between $\Aut(\pi)$ and $\mc{M}_{g,1}$  as it is a subgroupoid of $\Aut(\pi)$ while at the same time  $\mc{M}_{g,1}$ is a quotient of $\mf{CS}_{g,1}$.  

This survey essentially summarizes the results of \cite{bene1} and \cite{bene2}, as well as some results of \cite{abp}.  However, the perspective taken here is different in that we do not emphasize the use of fatgraphs and chord diagrams, but rather choose to focus on the dual notions of triangulations and polygon domains of $S_{g,1}$.  In some ways, this perspective is the most natural and classical, and hopefully this will allow for these results to be accessible to a wider audience. 

%%%%%%%%%%%%%%%%%%%%%
\section{The chord slide groupoid}\label{sect:CS}

Instead of considering all generating sets of $\pi$, let us consider only those which can be topologically realized by $2g$ disjoint arcs in $S_g$ based at $p$.  We will call such a set a CG set\footnote{These initials stand for \emph{combinatorial generating set} as introduced in \cite{bene2}, although perhaps   \emph{topological generating set} may have been a better name.} and consider two CG sets equivalent if they are realized by the same collection of arcs.  Cutting along any such collection $\mb{X}=\{x_i\}_{i=1}^{2g}$ of arcs decomposes $S_g$ into a polygon $P_{\mb{X}}$ with $4g+1$ edges labelled by $\pi$ (see below), $4g$ of which are identified in pairs.  We call this a \emph{polygon domain} of $S_{g,1}$, and there is a 1-1 correspondence between (equivalence classes of) CG sets and polygon domains of $S_g$.  

When a particular polygon domain is assumed, we shall denote its oriented sides by $\{c_i\}_{i=0}^{4g}$, where the ordering and orientation is given by the  clockwise cyclic ordering of $\partial P$ with  $c_0=\partial S_{g,1}$.  See Figure \ref{fig:fan}.  We shall often abuse notation and confuse an oriented  side of $P$ with the corresponding oriented arc of a CG set or generator in $\pi$, so that we will often simply write $c_i\in \pi$.   An immediate observation is that for any polygon $P$ decomposition of $S_{g,1}$ with sides $\{c_i\}_{i=0}^{4g}$, we have
\begin{equation}\label{eq:product}
\prod_{i=0}^{4g} c_i=1, \quad \textrm{thus} \quad \prod_{i=1}^{4g} c_i = \overline{\partial S_{g,1}}, 
\end{equation}
where again we consider the boundary $\partial S_{g,1}\in \pi$ as an element of $\pi$ and we use a bar to denote an inverse in $\pi$ or the reversal of an orientation of an arc.  As each arc of a CG set $\mb{X}$  corresponds to two oriented sides of  $P_\mb{X}$,  we can (and will) canonically orient and order the arcs of  $\mb{X}$ according to their first appearance in $\partial P$. %We shall assume from now on that all CG sets are so chosen.  

As a familiar example, a symplectic generating set $\{\alpha_i,\beta_i\}_{i=1}^{g}$ with $\prod_{i=1}^g [\alpha_i,\beta_i]=\partial S_{g,1}$ defines a polygon domain of $S_{g,1}$ with sides $c_1=\beta_{g}$, $c_2=\alpha_{g}$, $c_3=\bar \beta_g$, etc., identified in the ``standard'' way.  %,  such that  $\prod_{i=g}^1 [\beta_i,\alpha_i]=\overline{\partial S_{g,1}}$, where $ [\beta_i,\alpha_i]=\beta_i\alpha_i\bar \beta_i\bar\alpha_i$ .  

%Perhaps the most familiar example of a CG set is given by  a symplectic generating set $\mc{S}=\{\alpha_i,\beta_i\}_{i=1}^{2g}$ with 
%$
%\prod_{i=1}^{2g} [\alpha_i,\beta_i]=\partial S_{g,1},
%$,
%where $[\alpha_i,\beta_i]=\alpha_i\bar \beta_i\bar \alpha_i\beta_i$.  The sides of the corresponding polygon domain $P_{\mc{S}}$ would 

% For this, it will be convenient to choose a particular symplectic generating set $\mc{S}=\{\alpha_i,\beta_i\}_{i=1}^{2g}$, meaning 
%$
%\prod_{i=1}^{2g} [\alpha_i,\beta_i]=\partial S_{g,1},
%$
%to serve as our fixed ``basepoint''  in $\mf{CS}_{g,1}$ (here  $[\alpha_i,\beta_i]$ denotes the commutator of $\alpha_i$ and $\beta_i$).   Let $P_{\mc{S}}$ be the corresponding polygon domain of $S_{g,1}$, and let us denote its sides by $\{\sigma_i\}_{i=0}^{4g+1}$ so that  $\sigma_1=\beta_{2g}$, $\sigma_2=\alpha_{2g}$, $\sigma_2=\overline \beta_{2g}$, etc.  

Recall that a groupoid can be described as a category in which every morphism is an isomorphism.  
We define the \emph{chord slide groupoid}, denoted $\mf{CS}_{g,1}$, to be  the groupoid whose objects are copies of $\pi$, one for every equivalence class of CG set, and whose morphisms are the automorphisms of $\pi$ provided by taking one (canonically ordered) CG set to another.  
In this way, we have a groupoid morphism (functor)  $\mf{CS}_{g,1}\ra \Aut(\pi)$ which is neither 1-to-1 nor onto (neither faithful nor full), but does have trivial kernel.

% In this way, we can view  $\mf{CS}_{g,1}$ as a subgroupoid of $\Aut(\pi)$ (upon fixing a ``basepoint'' CG set for $\pi$).  
%\footnote{For the worried reader's benefit, we mention that  a more intuitive description in terms of generators and relations will be given later, which will shed some light on the groupoid's name.}  

 By the fundamental result of decorated \Teich theory, it is  known that $\mf{CS}_{g,1}$ is equivalent to a trivial groupoid, as it represents a discrete version of the fundamental path groupoid of a contractible space, the so-called decorated \Teich space of $S_{g,1}$ (see \cite{abp,penner,penner04}).  In other words, there is a unique morphism between any two objects of $\mf{CS}_{g,1}$.   Moreover, the mapping class group $\mc{M}_{g,1}$  acts freely on the chord slide groupoid via its action on (isotopy classes of) collections of arcs. Thus, the quotient $\mf{CS}_{g,1}/\mc{M}_{g,1}$ is a groupoid equivalent (as a groupoid) to the mapping class group $\mc{M}_{g,1}$ itself.  The advantage of the groupoid viewpoint presented here is that the generators and relations can be stated quite simply, as we shall see. 
    
\section{CS moves}

As the source and target of  every morphism of $\mf{CS}_{g,1}$ both correspond to a particular generating set for $\pi$, it makes sense to ask when a morphism of $\mf{CS}_{g,1}$ is an elementary Nielsen transformation.  Instead, let us ask the related question, when  a morphism  is a product of elementary transformations involving 
%For our purposes let us  generalize the definition of an elementary Nielsen transformation to be any automorphism which is a product of elementary Nielsen transformations involving 
only one multiplication, so that up to permutation we have
\[
x_i\mapsto (x_i^{\pm1}x_j^{\pm 1})^{\pm 1} 	\quad \textrm{for some $i\neq j, $} \quad x_k\mapsto x_k^{\pm1} \quad \textrm{for $k\neq i,j$}.
\]

    %######################################################
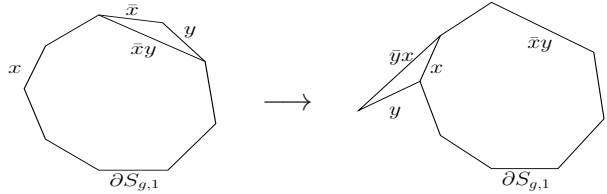
\begin{figure}[!h]
\begin{center}
\[  \begin{array}{c}\resizebox{1.1in}{1in}{\input{cut2.pstex_t}}\end{array} \quad  \lra  \quad   \begin{array}{c}\resizebox{1.3in}{1in}{\input{slide2.pstex_t}}\end{array}  \]
%\epsffile{images/Wcases.eps}
\caption{Triangle cut slide move.}
\label{fig:cutslide}
\end{center}
\end{figure}

It is easy to see when the answer to this new question is ``yes''.  It is exactly  when a  morphism is between CG sets which differ by the exchange of a single generator such that  the corresponding  polygon domains $P_1$ and $P_2$ are related  by cutting a triangle off of $P_1$ and reattaching it to another side of $P_1$ (according to the identification of sides of $P_1$).   See Figure \ref{fig:cutslide} where we depict the transformation $ x \mapsto \bar y x$.   We call such a move a \emph{triangle cut-slide move}, or simply a CS move for short. 

 It is not hard to see that a more general cut-slide move, where we cut off a larger polygon from $P_1$ and reattach it to another side, can be decomposed into triangle cut-slide moves. 
   In fact, we shall  soon see that CS moves generate all of $\mf{CS}_{g,1}$, but first we return to the topic of Nielsen reduction.

%%%%%%%%%%%%%%%%%%%%%%%%%%%%
\section{Nielsen reduction}

We now  adapt Nielsen reduction to the case of polygon domains of $S_{g,1}$, with CS moves taking the place of elementary Nielsen transformations.  For this, it will be convenient to fix a particular symplectic generating set $\mc{S}=\{\alpha_i,\beta_i\}_{i=1}^{2g}$ 
to serve as our ``basepoint''  in $\mf{CS}_{g,1}$.  Let $P_{\mc{S}}$ be the corresponding polygon domain of $S_{g,1}$, and let us denote its (ordered) sides by $\{\sigma_i\}_{i=0}^{4g+1}$, so that $\sigma_1=\beta_g$, etc.% so that  $\sigma_1=\beta_{2g}$, $\sigma_2=\alpha_{2g}$, $\sigma_2=\overline \beta_{2g}$, etc.  

Now consider any other CG set with polygon domain $P$ and sides $\{c_i\}_{i=0}^{4g}$.  Our goal is to find a sequence of CS moves which will transform $\{c_i\}_{i=1}^{4g}$ to $\{\sigma_i\}_{i=1}^{4g}$.  As with classical Nielsen reduction, we first concentrate on word length and define the \emph{length} of $P$ by 
\[
\abs{P}=\sum_{i=1}^{4g} \abs{c_i}
\]
where $\abs{x}$ denotes the word length of $x\in \pi$ with respect to the letters $\{\sigma_i\}_{i=1}^{4g}$.  

We call a side $c_i$ of $P$ \emph{unbalanced} if more than half of it cancels with one of its neighbors, meaning either $c_i=\ell \bar y$ and $c_{i+1}= y r$ as reduced words with $\abs{y}>\abs{\ell}$, or  $c_{i-1}=\ell \bar x$ and $c_i=xr$ as reduced words with $\abs{x}>\abs{r}$.  It is immediate that if $c_i$ is unbalanced, a  CS move involving $c_i$ can reduce the word length of $P$.  
As $\abs{P}=4g$ if and only if $P=P_\mc{S}$ (due to the restriction \eqref{eq:product}), if we were always able to find an unbalanced side of $P$, we would have our algorithm for evolving $P$ into $P_\mc{S}$.  
%If we could always  find a sequence of $n$ triangle cut-slide moves 
%\[
%P=P_0 \ra P_1\ra P_2 \ra \dotsm \ra P_n=P_s
%\]
%such that $\abs{P_i}<\abs{P_j}$ for $i>j$, then we would have our algorithm. 
 Unfortunately, we cannot always do this; however, we will always be able to find a CS move which makes progress towards $P_\mc{S}$ and is non-increasing in  length.
 
% We say that a side $c_i$ of $P$ \emph{cancels with its neighbors} if $c_{i-1}=\ell x$, $c_i=x\bar y$, and $c_{i+1}=yr$ as reduced words with respect to the letters $\sigma_j$.  Moreover, w
 
 We say that $c_i$ is balanced if $c_i=x\bar y$, $c_{i-1}=\ell \bar x$, and $c_{i+1}= y r$ as a reduced words with $\abs{x}=\abs{y}$, so that $c_i$ cancels equal amounts to the left and to the right.  Our choice of CG set $\mc{S}$ was motivated by the following fact (which we conjecture to be true for any CG set), whose proof we leave to the interested reader.
 \begin{lemma}
 If $P$ has no unbalanced or balanced side,  then $P=P_\mc{S}$.  
 \end{lemma}
 
 %Now, assume we have some side $c_i$ of $P$ which cancels with its neighbors.  If $c_i$ is not balanced, say $\abs{x}<\abs{y}$, then we reduce the length of $P$ by cutting off the triangle $\tau_{i,i+1}$ defined by $c_i$ and $c_{i+1}$ and reattaching it at $\overline c_{i+1}$. 

%\begin{lemma}
%If $P\neq P_s$ is a polygon domain in which no triangle cut-slide move decreases the length $\abs{P}$, then there exists some side $c_i$ of $P$ which cancels equally with its two neighbors, meaning $c_i=xy$ with $\abs{x}=\abs{y}$, $c_{i-1}=\ell \bar x$, and $c_{i+1}=\bar y r$ as reduced words.  
%\end{lemma}
%\begin{proof}
%The technical part of the proof relies on basic cancellation theory and the form (combinatorics) of the symplectic generators $\{\alpha_i,\beta_i\}_{i=1}^{2g}$ to show that under the assumptions of the lemma, we must have some $c_i$ which is completely cancelled by its two neighbors.  Since no triangle cut-slide move reduces the word length, this $c_i$ must cancel equally on both sides.  
%\end{proof}

Now, to make use of the above lemma, we need to introduce an \emph{energy function}  which lexicographically orders $\pi$.  % lexicographical ordering on $\pi$ defined by an ``energy'' function.  
We define the \emph{energy} $\norm{x}$ of an element $x\in \pi$ by $\norm{\sigma_i}=i$ for $1\leq i \leq 4g$ and 
\[
\norm{w}=\norm{\prod_{j=1}^{\abs{w}}  \sigma_{i_j}}=\sum_{j=1}^{\abs{w}} (4g+1)^{\abs{w}-j}\norm{\sigma_{i_j}}
\]
for $w=\prod_{j=1}^{\abs{w}}  \sigma_{i_j}$ as a reduced word. 
Note that this indeed defines a lexicographical ordering on $\pi$ and that it extends the word length function in that $\abs{x}<\abs{y}$ implies $\norm{x}<\norm{y}$. We similarly define the energy of a polygon domain by 
\[
\norm{P}=\sum_{i=1}^{4g}\norm{c_i}.
\]

\begin{lemma}%[\cite{bene2}]
If $c_i$ is a balanced side of $P$, then a triangle cut-slide move involving $c_i$ reduces the energy $\norm{P}$ of $P$.
\end{lemma}
\begin{proof}
Let $c_i=x\bar y$ with $\abs{x}=\abs{y}$, $c_{i-1}=\ell \bar x$, and $c_{i+1}=y r$.  
If we cut off the triangle $\tau_{i-1,i}$ defined by $c_i$ and $c_{i-1}$ and reattach it to the side $\overline c_{i-1}$, then the CG set is changed by $\ell \bar x\mapsto \ell \bar y$ (and thus also $x\bar \ell \mapsto y\bar \ell$).  Similarly, if we cut off $\tau_{i,i+1}$ and attach it to $\overline c_{i+1}$, then we have $yr \mapsto xr$ (and also $\bar r \bar y \mapsto \bar r \bar x$).  Thus it is not hard to see that if 
$\norm{x}<\norm{y}$, then the first cut-slide move reduces the energy, while if $\norm{x}>\norm{y}$ then the second one does.  As one of these two inequalities must hold, we have our result.
\end{proof}

Collecting the above lemmas, we are able to describe our ``fatgraph'' Nielsen reduction algorithm: whenever possible, reduce the word length of $P$ by a CS move involving an unbalanced side.  When no unbalanced side exists, perform a CS move on a balanced side to reduce the energy.

\section{Presentation of  $\mf{CS}_{g,1}$}

Recall that a groupoid can also be defined as a partially composable set with inverses.  Viewing $\mf{CS}_{g,1}$ in this way, we see that a consequence of the fatgraph Nielsen reduction is the following lemma:
\begin{lemma}\label{lem:gen}
Triangle cut-slide moves generate $\mf{CS}_{g,1}$.
\end{lemma}
In this section, we shall give an alternative proof of this, as well as a description of the relations of $\mf{CS}_{g,1}$.  To this end, we begin by noting that any polygon domain $P$ of $S_{g,1}$ can be canonically extended to a triangulation $T(P)$ of $S_{g,1}$ (based at $p$) by triangulating the polygon in a fan-like fashion as depicted in  Figure \ref{fig:fan}a.

%######################################################
\begin{figure}[!h]
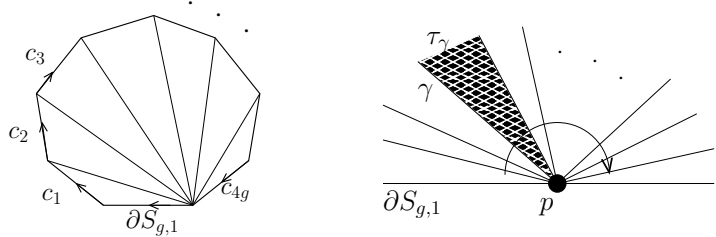

\begin{center}
\[  \fttexw{1.33in}{fan} \quad \quad \quad  \fttexw{1.75in}{arcsatp}   \]
%\epsffile{images/Wcases.eps}
\caption{a) Fan triangulation and b) Arcs based at $p$ with $\tau_\gamma$ to the right of $\gamma$.}
\label{fig:fan}
\end{center}
\end{figure}

Conversely, an inverse to this fan-like triangulation is given by the following greedy algorithm \cite{abp}:  Given a triangulation $T$ of $S_{g,1}$ based at $p$,  we canonically order the arcs of $T$  according to their first appearance in the clockwise ordering at $p$ (see Figure \ref{fig:fan}b).
In this ordering, we consecutively remove every arc from this collection as long as the compliment of the remaining arcs  in $S_{g,1}$  consists of a union of polygons.  This is clearly an inverse of the fan-like triangulation.  Moreover, we have the following ``locality'' result: %, but what is perhaps not so immediate is the following ``locality'' result.
\begin{lemma}\label{lem:locality}
Given an arc $\gamma$ in a triangulation $T$, let $\tau_\gamma$ be the triangle lying to the right of $\gamma\in T$  at its first occurrence in the ordering of arcs at $p$ (see Figure \ref{fig:fan}b), and let $v$ be the sector of $\tau_\gamma$ opposite to $\gamma$.  Then $\gamma$ is removed during the greedy algorithm if and only if $v$ precedes $\gamma$ in the clockwise order at $p$. 
\end{lemma}
\begin{proof}
Whenever an arc is removed during the greedy algorithm, its complement in $S_{g,1}$ changes by attaching a  triangle to  the current polygon containing $\partial S_{g,1}$ as a side.  This can happen for $\gamma$ if and only if $v$ does not precede $\gamma$.   
\end{proof}

%######################################################
\begin{figure}[!h]
\begin{center}
\includegraphics[width=1.8in]{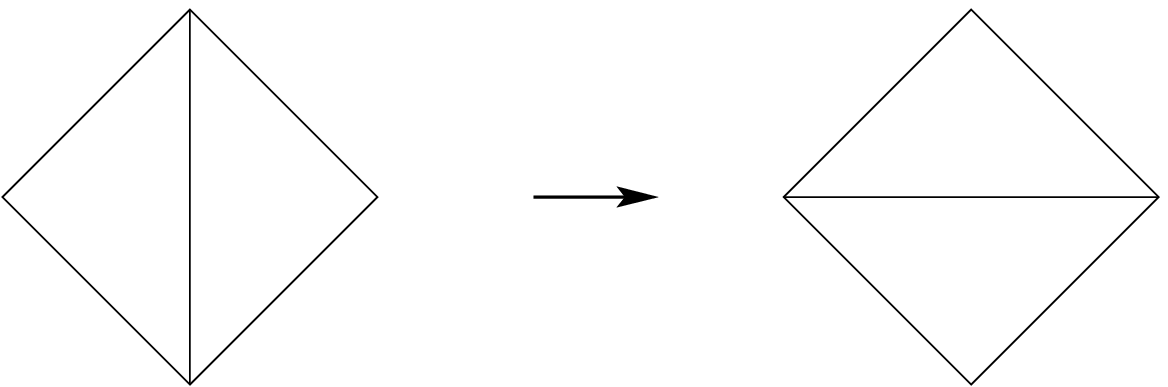} $\quad\quad\quad $  \includegraphics[width=1.7in]{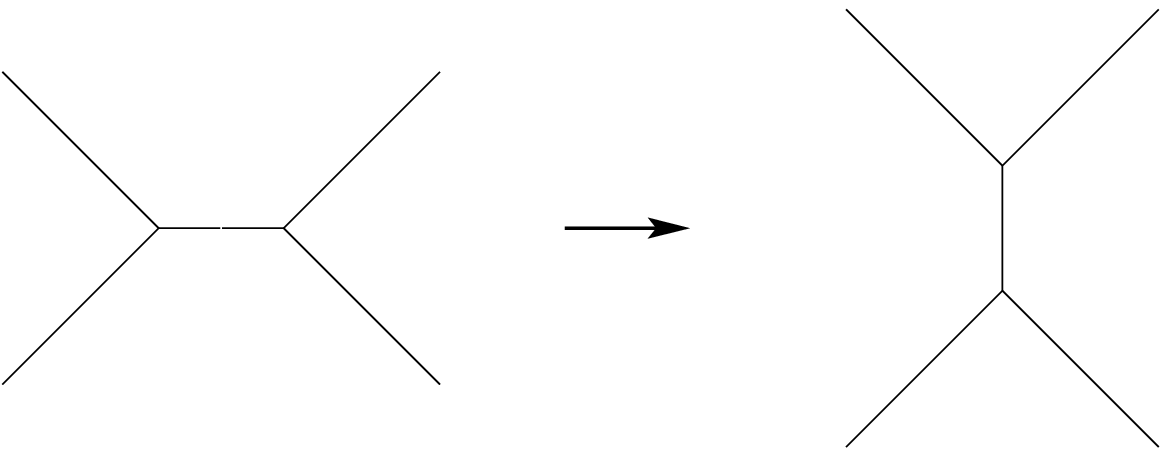}
%\epsffile{images/Wcases.eps}
\caption{a) A Diagonal exchange. b)  A Whitehead move.}
\label{fig:diagx}
\end{center}
\end{figure}

We are now ready to give our alternate proof of Lemma \ref{lem:gen}.
\begin{proof}[Alternate proof of Lemma \ref{lem:gen}]
%We shall give a very casual proof of this fact, as it shall be a consequence of the fatgraph Nielsen reduction algorithm described later.  Here we rely  on 
The proof relies on  the classical result of Whitehead which states that any two triangulations of $S_{g,1}$ are related by a sequence of  elementary diagonal exchanges (see Figure \ref{fig:diagx}a), where an arc is removed and replaced by the opposite diagonal in the resulting quadrilateral.  
Note that a triangle cut-slide move can be realized by (usually) two or (sometimes) one diagonal exchange on the corresponding fan-like triangulation of $S_{g,1}$.  
We essentially need  to show the converse.  

Given any two polygon domains $P$ and $P'$ of $S_{g,1}$, we can canonically triangulate each to obtain triangulations $T$ and $T'$. By Whitehead's result, $T$ and $T'$ are related by a sequence of diagonal exchanges: $T=T_0\ra T_1\ra \dotsm \ra T_n=T'$.  Let $P_i=P(T_i)$ denote the $i$th  polygon domain corresponding to $T_i$ via  the greedy algorithm so that  $P_0=P$ and $P_n=P'$.  Lemma \ref{lem:locality} can be interpreted as saying the inclusion of an arc in $T_i$ as a side of $P_i$ is a ``local'' property; thus, as $T_{i+1}$ differs from $T_{i}$ by the replacement of a single arc, $P_{i+1}$ can differ from $P_i$ only by the replacement of a single arc.  
 %sides  polygon $P_i$ is determined ``locally'' by the arcs of $T_i$; 
% as a consequence,   $P_i$ is related to $P_{i+1}$ by replacing a single side of $P_i$ with a new arc.
As we have already noted,  such a move can be realized by a sequence of CS moves;  thus,  we have our result.
\end{proof}

Once we know how to generate $\mf{CS}_{g,1}$, it is natural to ask what relations the groupoid satisfies.  %An obvious relation would consist of 
Some relations are immediate, such as following a CS move by its inverse, which reattaches the cut triangle to its original position.  
%There are some obvious relations such as 
%performing 
%a triangle cut-slide move followed by another which cuts off the same triangle and reattaches it to its original position.
% and then cutting off the same triangle and attaching it back to its original position.  
We call this the \emph{involutivity} relation $I$.  Similarly, cut-slide moves for  non-adjacent triangles (which remain non-adjacent after sliding) are easily seen to commute with each other, and we call this the \emph{commutativity} relation $C$.  Also, it is not hard to see that we have the following   \emph{triangle} relation $T$ where a triangle $\tau_{x,y}$ with sides $x$ and $y$ is cut off and attached to $\bar x$, then cut off again and attached to $\bar y$, then cut off again and reattached in its original position.

%######################################################
\begin{figure}[!h]
\begin{center}
\[
 \fttw{width=0.93in}{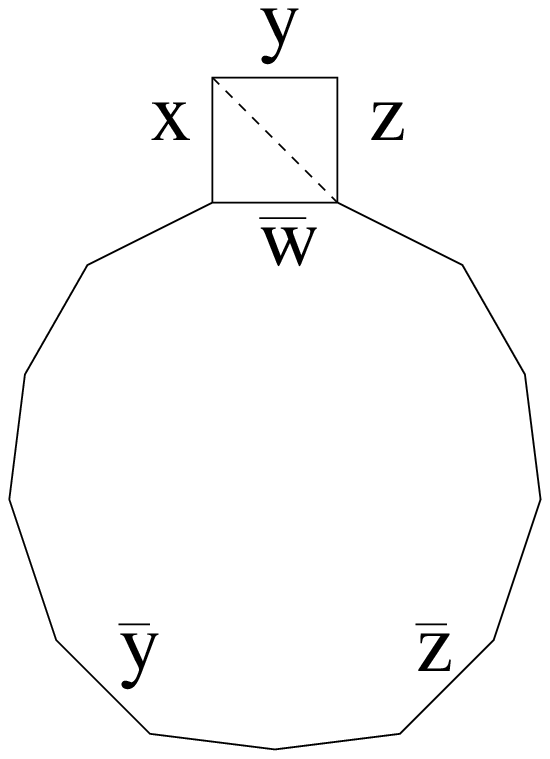}
\ra \fttw{width=0.93in}{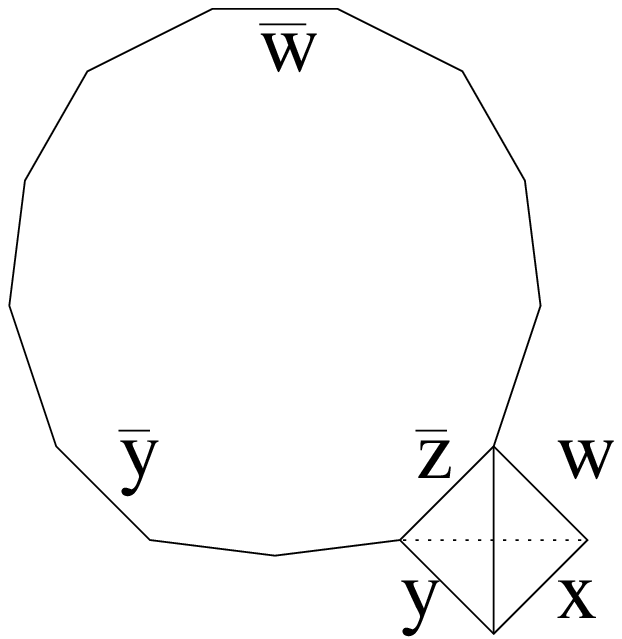}
\ra \fttw{width=0.93in}{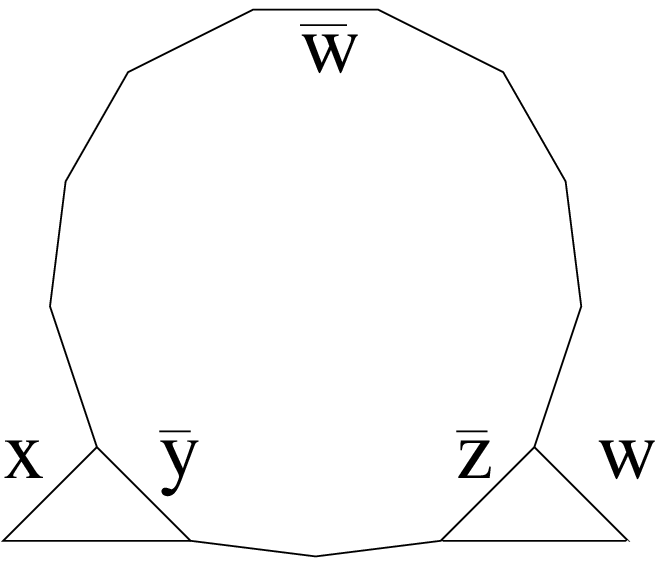}
\ra \fttw{width=0.93in}{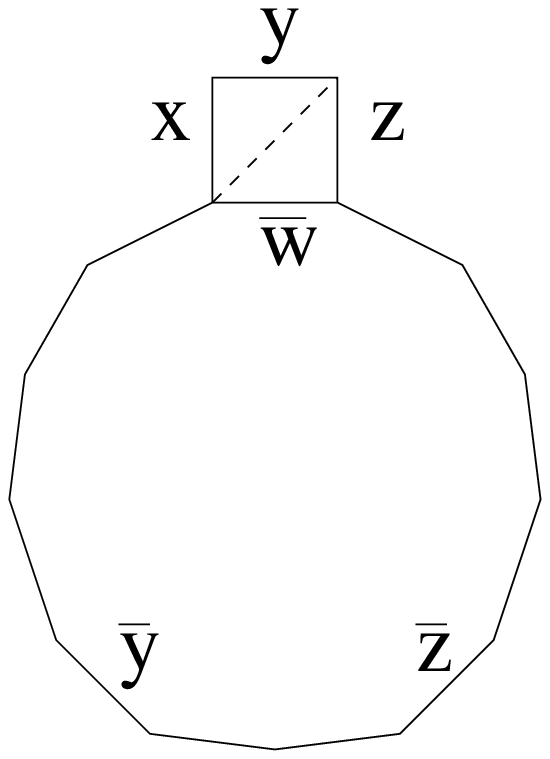}
\]
%\epsffile{images/Wcases.eps}
\caption{Left pentagon relation.}
\label{fig:L}
\end{center}
\end{figure}

Finally, we have two relations  which involve cutting and sliding two adjacent triangles.  The \emph{left pentagon} relation $L$ is depicted in Figure \ref{fig:L}.  The \emph{right pentagon}  $R$ is defined analogously. 

%We call them the  left and right pentagon relations, $L$ and $R$.  $L$ is depicted in Figure \ref{fig:L} and $R$ is defined analogously.  

\begin{theorem}[\cite{bene1}]
$I$, $C$, $T$, $L$, and $R$ generate all relations in $\mf{CS}_{g,1}$.  
\end{theorem}
%######################################################
\begin{figure}[!h]
\begin{center}
\[
 \fttw{width=0.55in}{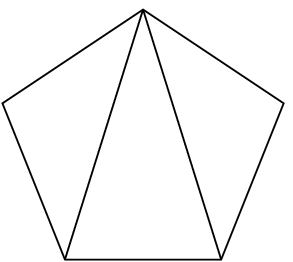}
\ra \fttw{width=0.55in}{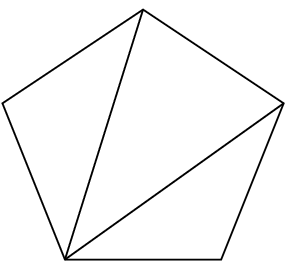}
\ra \fttw{width=0.55in}{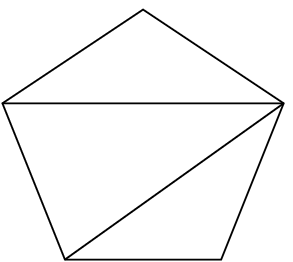}
\ra \fttw{width=0.55in}{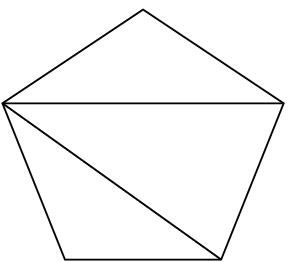}
\ra \fttw{width=0.55in}{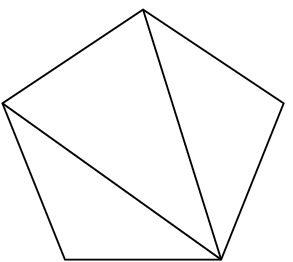}
\ra \fttw{width=0.55in}{P1.eps}
\]
%\epsffile{images/Wcases.eps}
\caption{Pentagon relation.}
\label{fig:pentagon}
\end{center}
\end{figure}
\begin{proof}
The proof is a bit technical and involves many cases, but the main idea is simple.  It relies on the results of decorated \Teich theory \cite{penner} which tell us precisely what the relations are for the  analogous groupoid generated by diagonal exchanges, the so-called \emph{Ptolemy groupoid} $\mf{Pt}_{g,1}$ \cite{Penner03}.  All relations in $\mf{Pt}_{g,1}$ are generated by three types of relations: the obvious involutivity relation where a diagonal exchange is followed by its inverse, the commutativity relation involving diagonal exchanges on non-adjacent arcs, and the famous \emph{pentagon relation} which is depicted in Figure \ref{fig:pentagon}.  

The proof then proceeds by doing a careful analysis between diagonal exchanges and cut-slide moves via the greedy algorithm to rewrite (the image under the greedy algorithm of) each possible incarnation of a relation in $\mf{Pt}_{g,1}$ as a product of the relations $I$, $C$, $T$, $L$, and $R$.
\end{proof}

\section{Fatgraphs and Chord Diagrams}

While we have so far focused on the perspective of arc systems and polygon domains, there is a dual perspective which is worth briefly mentioning (and in fact is the perspective taken in \cite{abp,bene1,bene2}).  Given any (non-degenerate) arc system $Y=\{y_i\}$ based at $p$ which cuts $S_{g,1}$ into a number of polygon components, we define the \Poin dual graph $G_Y$ of $Y$  %as follows.  First, if not already present, add a pushed-off copy of the boundary to $Y$ to obtain $\widehat Y=Y\cup \partial S_{g,1}$.  
to be the graph embedded in $S_{g,1}$ which has one vertex for every component of $S_{g,1}\backslash Y$ %and one vertex  on the boundary $q\neq p \in \partial S_{g,1}$; 
and  one edge $e_i$ for every arc $y_i$ of $Y$ such that $e_i$ intersects $y_j$ if and only if  $i=j$. % and one edge from . 
 In particular, if $T$ is a triangulation, $G_T$ is a trivalent graph, and if $X$ is a CG set, $G_X$ is a ``rose'' graph with only one vertex.  Note that every oriented edge of the  graph $G_Y$ is colored by an element of $\pi$ in a natural way (after  choosing an  orientation of $S_{g,1}$).

In fact, this dual graph inherits some  additional structure, as the orientation of the surface induces a cyclic ordering of half-edges incident to every vertex.  We call such a vertex-oriented graph a  \emph{fatgraph}.  It is convenient (for technical reasons)  to consider the boundary $\partial S_{g,1}$ to be an arc included in every arc system, in which case  \Poin duality will produce \emph{bordered fatgraphs}:  fatgraphs with a special edge called the \emph{tail} whose univalent endpoint lies on the boundary   $t\neq p \in \partial S_{g,1}$.  (See Figure \ref{fig:chordslide}.) 
%######################################################
\begin{figure}[!h]
\begin{center}
\[\; \quad \fttexw{1.4in}{cd0} \; \quad  \ra \;  \quad \fttexw{1.6in}{cd5}   \]  
\[
 \fttw{width=2.25in}{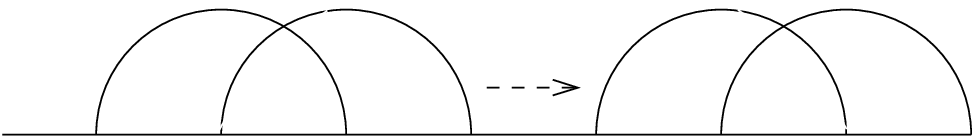} \ra   \fttw{width=2.25in}{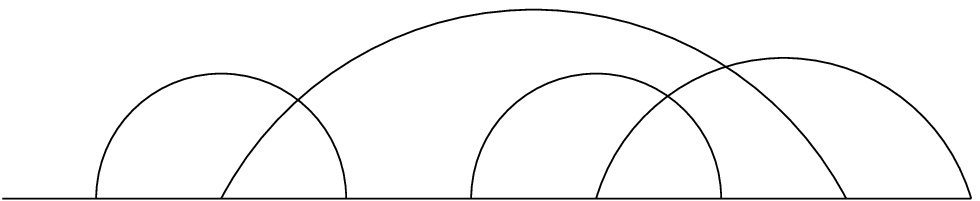} 
\]
%\epsffile{images/Wcases.eps}
\caption{Chord diagram  and corresponding  fan-like triangulation evolving under a chord slide.}
\label{fig:chordslide}
\end{center}
\end{figure}

We are primarily interested in trivalent bordered fatgraphs which are dual to triangulations of $S_{g,1}$.  In particular, the graph we obtain as the dual of a fan-like triangulation of a polygon domain takes the form of a \emph{linear chord diagram}.   A linear chord diagram is a trivalent graph immersed in the plane consisting of a segment of the real axis called the \emph{core}, together with line segments lying in the upper half-plane called the \emph{chords}  which are attached to distinct points of the core.  %which are attached to distinct points of the core and lie in the upper half plane.
  Under this duality, the chords of a linear chord diagram embedded in $S_{g,1}$ exactly correspond to pairs of identified sides of a polygon domain of $S_{g,1}$, and the structure of the chord diagram essentially captures the information of how these sides are identified.  

We now depict how the elementary moves of  arc systems look in this dual viewpoint.  Firstly, a diagonal exchange on a triangulation corresponds to a \emph{Whitehead move} on a trivalent bordered fatgraph, which is a move where one non-tail edge is collapsed to a four-valent vertex and then expanded in the opposite direction.  See \cite{abp}.
A triangle CS move, on the other hand, corresponds to a \emph{chord slide} (thus, finally explaining the name of the chord slide groupoid) on a linear chord diagram, which is a move where one endpoint of a chord is slid along a neighboring chord to a new position on the core, as depicted in Figure \ref{fig:chordslide}.  During both of these moves, $\pi$-colorings of the edges of these graphs evolve in natural ways, according to certain ``vertex compatibility'' relations.  See \cite{abp, bene1}.  
% the labelings of the oriented edges of the graphs by elements of $\pi$ evolve in a simple manner.  

For convenience, we finish by listing   diagrammatically  the $T$, $L$, and $R$ relations for linear chord diagrams, where the thin lines represent chords and the thick lines represent segments of the core.

%######################################################
\begin{figure}[!h]
\begin{center}
\includegraphics[width=3.7in]{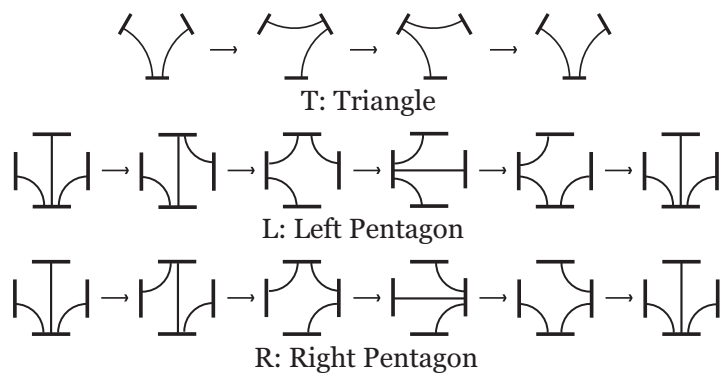}
%\epsffile{images/Wcases.eps}
\caption{$T$, $L$, and $R$ relations for chord diagrams.} 
\label{fig:somerelations}
\end{center}
\end{figure}

\bibliographystyle{amsplain}

\end{document}

%% file: cut2.pstex_t
\begin{picture}(0,0)%
\includegraphics{cut2.pstex}%
\end{picture}%
\setlength{\unitlength}{3947sp}%
\begingroup\makeatletter\ifx\SetFigFont\undefined%
\gdef\SetFigFont#1#2#3#4#5{%
  \reset@font\fontsize{#1}{#2pt}%
  \fontfamily{#3}\fontseries{#4}\fontshape{#5}%
  \selectfont}%
\fi\endgroup%
\begin{picture}(2962,3690)(3436,-9859)
\put(3451,-7561){\makebox(0,0)[lb]{\smash{{\SetFigFont{20}{24.0}{\rmdefault}{\mddefault}{\updefault}{\color[rgb]{0,0,0}$x$}%
}}}}
\put(5176,-7186){\makebox(0,0)[lb]{\smash{{\SetFigFont{20}{24.0}{\rmdefault}{\mddefault}{\updefault}{\color[rgb]{0,0,0}$\bar xy$}%
}}}}
\put(4876,-9736){\makebox(0,0)[lb]{\smash{{\SetFigFont{20}{24.0}{\rmdefault}{\mddefault}{\updefault}{\color[rgb]{0,0,0}$\partial S_{g,1}$}%
}}}}
\put(5926,-6811){\makebox(0,0)[lb]{\smash{{\SetFigFont{20}{24.0}{\rmdefault}{\mddefault}{\updefault}{\color[rgb]{0,0,0}$y$}%
}}}}
\put(5101,-6436){\makebox(0,0)[lb]{\smash{{\SetFigFont{20}{24.0}{\rmdefault}{\mddefault}{\updefault}{\color[rgb]{0,0,0}$\bar x$}%
}}}}
\end{picture}%

%% file: slide2.pstex_t
\begin{picture}(0,0)%
\includegraphics{slide2.pstex}%
\end{picture}%
\setlength{\unitlength}{3947sp}%
\begingroup\makeatletter\ifx\SetFigFont\undefined%
\gdef\SetFigFont#1#2#3#4#5{%
  \reset@font\fontsize{#1}{#2pt}%
  \fontfamily{#3}\fontseries{#4}\fontshape{#5}%
  \selectfont}%
\fi\endgroup%
\begin{picture}(3644,3445)(2004,-9109)
\put(2476,-7711){\makebox(0,0)[lb]{\smash{{\SetFigFont{20}{24.0}{\rmdefault}{\mddefault}{\updefault}{\color[rgb]{0,0,0}$y$}%
}}}}
\put(2476,-6736){\makebox(0,0)[lb]{\smash{{\SetFigFont{20}{24.0}{\rmdefault}{\mddefault}{\updefault}{\color[rgb]{0,0,0}$\bar yx$}%
}}}}
\put(4501,-6511){\makebox(0,0)[lb]{\smash{{\SetFigFont{20}{24.0}{\rmdefault}{\mddefault}{\updefault}{\color[rgb]{0,0,0}$\bar xy$}%
}}}}
\put(4126,-8986){\makebox(0,0)[lb]{\smash{{\SetFigFont{20}{24.0}{\rmdefault}{\mddefault}{\updefault}{\color[rgb]{0,0,0}$\partial S_{g,1}$}%
}}}}
\put(3076,-6961){\makebox(0,0)[lb]{\smash{{\SetFigFont{20}{24.0}{\rmdefault}{\mddefault}{\updefault}{\color[rgb]{0,0,0}$x$}%
}}}}
\end{picture}%